\documentclass{article}
\usepackage{amsmath,amsfonts,amsthm,comment,url}
\usepackage{graphicx}
\usepackage{hyperref}

\theoremstyle{plain}
\newtheorem{theorem}{Theorem}[section]

\newtheorem{lemma}[theorem]{Lemma}

\theoremstyle{definition}
\newtheorem{definition}[theorem]{Definition}
\newtheorem{example}[theorem]{Example}

\newcommand{\EE}{\mathbb{E}}
\newcommand{\ee}{\mathrm{e}}
\newcommand{\ff}[1]{{\underline{#1}}}

\newcommand{\ii}{\mathrm{i}}

\newcommand{\Po}{\operatorname{Po}}

\newcommand{\Bin}{\operatorname{Bin}}

\newcommand{\Disp}{\operatorname{Disp}}
\newcommand{\Var}{\operatorname{Var}}
\newcommand{\NN}{\operatorname{N}}
\newcommand{\Herm}{\operatorname{Herm}}

\newcommand{\Sib}{\operatorname{Sib}}
\newcommand{\DSib}{\operatorname{DSib}}
\newcommand{\PoSib}{\operatorname{PoSib}}
\newcommand{\PoDSib}{\operatorname{PoDSib}}

\newcommand{\eqd}{\equiv}

\begin{document}

\title{Discrete broadly stable distributions}
\author{Matthew Aldridge}
\date{12 September 2025}

\maketitle

\begin{abstract}
A random variable $X$ is strictly stable if a sum of independent copies of $X$ has the same distribution as $X$ up to scaling, and is stable (in the broad sense) if the sum has the same distribution as $X$ up to both scaling and shifting. Steutel and Van Harn studied `discrete strict stability' for random variables taking values in the non-negative integers, replacing scaling by the thinning operation. We extend this to `discrete stability' (in the broad sense), where shifting is replaced by the addition of an independent Poisson random variable. The discrete stable distributions are `Poisson--delayed Sibuya' distributions, which include the Poisson and Hermite distributions. Similar results have been proved independently by Townes \cite{townes}.
\end{abstract}

\tableofcontents

\section{Stability and discrete stability}

\subsection{Definitions}

We start by recalling the definitions of strictly stable and stable real-valued random variables, before looking at their discrete equivalents. The concept of stability goes back to the work of Paul Lévy in the 1920s; see, for example, \cite[Chapters VI and XVII]{feller2}, \cite[Chapters 1 and 3]{nolan} 
for more detailed coverage than our brief summary here.

A random variable $X$ is strictly stable if a sum of independent copies of $X$ has the same distribution as $X$ up to scaling, and is stable (in the broad sense) if the sum has the same distribution as $X$ up to both scaling and shifting.

Throughout, we write $\eqd$ for equality in distribution.

\begin{definition} \label{def:rstab}
Consider a real-valued random variable $X$, and let $X_1, X_2 \dots, X_n$ be IID copies of $X$.

We say that $X$ is \emph{strictly stable} if for all $n$ there exists a constant $a_n \geq 0$ such that
\begin{equation} \label{eq:rstrict}
a_n (X_1 + X_2 + \cdots + X_n) \eqd X .
\end{equation}

We say that $X$ is \emph{stable} (in the \emph{broad} sense) if for all $n$ there exist constants $a_n \geq 0$ and $b_n$ such that
\begin{equation} \label{eq:rstab}
a_n (X_1 + X_2 + \cdots + X_n) \eqd X  + b_n .
\end{equation}
\end{definition}

However, none of the strictly stable or stable distributions are discrete distributions, with the exception of the degenerate point masses. So to get interesting results for discrete random variables we need to consider different definitions. In this paper, we consider definitions of `discrete stability' for discrete random variables, and specifically for `count' random variables that take values in the non-negative integers.

A discrete version of strict stability was studied by Steutel and Van Harn \cite{steutel1}, \cite[Section V.5]{steutel2}.
To get a discrete version, Steutel and Van Harn replaced the scaling operation of multiplying a real-valued $X$ by a scalar $a \geq 0$ to get $aX$ with the thinning operation of thinning a non-negative integer-valued $X$ by a probability $a \in [0,1]$ to get $a \circ X$. Recall that \emph{thinning} can be thought of as each of $X$ items being kept independently with probability $a$; more formally, the conditional distribution of $a \circ X$ given $X$ is Binomial$(X,a)$. We give more details about thinning in Subsection \ref{sec:shift}.

\begin{definition}[\cite{steutel1}] \label{def:dstrict}
Consider a discrete random variable $X$ taking values in the non-negative integers, and let $X_1, X_2 \dots, X_n$ be IID copies of $X$.

We say that $X$ is \emph{discrete strictly stable} if for all $n$ there exists a constant $a_n \in [0,1]$ such that
\begin{equation} \label{eq:dstrict}
a_n \circ (X_1 + X_2 + \cdots + X_n) \eqd X .
\end{equation}
\end{definition}

Steutel and Van Harn called this simply `discrete stable', but it's clear from comparing the definitions in \eqref{eq:rstrict} and \eqref{eq:dstrict} that `discrete \emph{strictly} stable' is more appropriate.

The goal of this paper is to define and study a discrete equivalent of stability in the broad sense.
To do this, we need a discrete equivalent of shifting a real-valued random variable $X$ by a constant $b \in \mathbb R$ to get $X + b$. We propose that a suitable equivalent is `Poisson shifting'.  For $b \geq 0$, we can `Poisson shift' a discrete random variable $X$ by adding an independent Poisson random variable with mean $b$. We will denote this $X \oplus b = X + Z$, where $Z \sim \Po(b)$ is independent of $X$. 
We will give more details about the Poisson shift -- including arguments in favour of the Poisson shift as the natural `discrete shift' -- later in Subsection \ref{sec:shift}.

We can now state our proposed new definition of discrete stability (in the broad sense).

\begin{definition} \label{def:dstab}
Consider a discrete random variable $X$ taking values in the non-negative integers, and let $X_1, X_2 \dots, X_n$ be IID copies of $X$.

We say that $X$ is \emph{discrete stable} (in the \emph{broad} sense) if for all $n$ there exist constants $a_n \in [0,1]$  and $b_n > 0$ such that
\begin{equation} \label{eq:dstab}
a_n \circ (X_1 + X_2 + \cdots + X_n) \eqd X \oplus b_n  .
\end{equation}
\end{definition}

\subsection{Theorems}

We now go over four theorems -- three old; one new here -- that classify the stable distributions for the four definitions of stability above.



We recall the following two foundational results on usual non-discrete stability. With the exception of some special cases, probability density functions for stable distributions cannot be written down in closed form, but the characteristic functions $\phi_X(t) = \EE \exp(\ii tX)$ can be. For full details (especially the characterisations in terms of characteristic functions, which we omit), see, for example, \cite[Chapters VI and XVII]{feller2}, \cite[Chapters 1 and 3]{nolan}.

\begin{theorem} \label{th:rstrict}
If $X$ is strictly stable, in that \eqref{eq:rstrict} holds, and $X$ is not a point mass at $0$, then $a_n = n^{-1/\alpha}$ for some $\alpha \in (0,2]$.

The strictly stable random variables can be fully characterised in terms of their characteristic functions $\phi_X$.
\end{theorem}

\begin{theorem} \label{th:rstab}
If $X$ is stable, in that \eqref{eq:rstab} holds, and $X$ is not a point mass, then $a_n = n^{-1/\alpha}$ for some $\alpha \in (0,2]$.

The stable random variables can be fully characterised in terms of their characteristic functions $\phi_X$.
\end{theorem}


Results for discrete random variables can be stated using the probability generating function $G_X(t) = \mathbb E \,t^X$. However, we will find it more convenient to use the \emph{alternate probability generating function} (APGF) $\psi_X(t) = G_X(1 - t) = \mathbb E(1 - t)^X$, as in \cite[Section 1.2.7]{UDD}, \cite{aldridge-thin}. Just as $G_X(t)$ always exists at least for $t \in [-1, 1]$, so the APGF $\psi_X(t) = G_X(1 - t)$ always exists for $t \in [0,2]$. 

Steutel and Van Harn's central result for discrete strict stability -- a discrete equivalent to Theorem \ref{th:rstrict} -- is the following \cite{steutel1} \cite[Section V.5]{steutel2}.

\begin{theorem}[\cite{steutel1}] \label{th:dstrict}
If $X$ is discrete strictly stable, in that \eqref{eq:dstrict} holds, and $X$ is not a point mass at $0$, then $a_n = n^{-1/\alpha}$ for some $\alpha \in (0,1]$.

The discrete strictly stable random variables have alternate probability generating function
\[ \psi_X(t) = \exp(-\gamma t^\alpha) \qquad \text{where } \gamma \geq 0 \qquad \text{for } \alpha \in (0, 1]. \]

Any function of this form is the alternate probability generating function of discrete strictly stable random variable; specifically, of a Poisson--Sibuya distribution.
\end{theorem}

The connection with the Sibuya distribution was first noted by Devroye \cite{devroye}. We discuss the Sibuya and Poisson--Sibuya distributions in Subsection \ref{sec:sib}.

While Theorems \ref{th:rstrict} and \ref{th:dstrict} are quite similar, there is one notable difference: the stability parameter $\alpha$ can be anywhere in $(0,2]$ for strict stability, but can only be in $(0,1]$ for discrete strict stability.

The main result of this paper, is the following result for discrete stability: a discrete equivalent of Theorem \ref{th:rstab}.

\begin{theorem} \label{th:dstab}
If $X$ is discrete stable, in that \eqref{eq:dstab} holds, and $X$ is not a Poisson distribution, then $a_n = n^{-1/\alpha}$ for some $\alpha \in (0, 2]$.

The discrete stable distributions have alternate probability generating function
\begin{equation} \label{eq:maineq}
\psi_X(t) = \begin{cases}
\exp (-\delta t - \gamma t^\alpha ) & \text{where } \gamma \geq 0, \,\delta \geq - \alpha \gamma \qquad \text{for } \alpha \in (0,1) \\
\exp (-\delta t - \gamma t\log t ) & \text{where } \gamma \leq 0,\, \delta \geq -\gamma \hspace{34.5pt} \text{for } \alpha = 1 \\
\exp (-\delta t - \gamma t^\alpha ) & \text{where } \gamma \leq 0, \,\delta \geq - \alpha \gamma \qquad \text{for } \alpha \in (1,2] .
\end{cases}
\end{equation}

Any function of this form is the alternate probability generating function of a discrete stable distribution; specifically, of a Poisson--delayed Sibuya distribution.
\end{theorem}

We discuss the delayed Sibuya and Poisson--delayed Sibuya distributions in Subsection \ref{sec:dsib}.

Note again the similarity of Theorems \ref{th:rstab} and \ref{th:dstab}; 
even better, unlike for discrete strict stability, we now have the same full range $(0, 2]$ for the stability parameter $\alpha$, as in the real-valued case. 

The condition in Theorem \ref{th:dstab} that $X$ is not Poisson should be seen as equivalent to the condition in Theorem \ref{th:rstab} that $X$ is not a point mass: in both cases there are multiple solutions $(a_n, b_n)$ to \eqref{eq:dstab} and \eqref{eq:rstab} respectively, where $a_n$ \emph{can} be, but doesn't \emph{have} to be, of the form $n^{-1/\alpha}$ for $\alpha \in (0,2]$.

\bigskip

\noindent Independently of this work, the preprint ``Broadly discrete stable distributions'' by F.~William Townes \cite{townes} covers the same topics as this paper and proves very similar results. 

\subsection{Examples}

To illustrate the definitions of stability, here are some examples of stable distributions.

\begin{example} \label{ex:examples}
\mbox{}
\begin{itemize}
\item An example of strict stability is the point mass, where $a_n = n^{-1}$. If $X$ is a point mass at $\mu$, then
\[ n^{-1}(X_1 + X_2 + \cdots + X_n) \]
is also a point mass at $\mu$.
\item An example of discrete strict stability is the Poisson distribution, where $a_n = n^{-1}$. If $X$ is a Poisson distribution with mean $\mu$, then
\[ n^{-1} \circ (X_1 + X_2 + \cdots + X_n) \]
is also a  Poisson distribution with mean $\mu$.
\item An example of stability is the normal distribution, where $a_n = n^{-1/2}$. If $X$ is a normal distribution with mean $\mu$ and variance $\sigma^2$, then
\[ n^{-1/2}(X_1 + X_2 + \cdots + X_n) \eqd X + (n^{1/2}-1)\mu \]
are both normal distributions with mean $n^{1/2}\mu$ and variance $\sigma^2$.
\item An example of discrete stability is the Hermite distribution (see Definition \ref{def:herm} later), where $a_n = n^{-1/2}$. If $X$ is a Hermite distribution with mean $\mu$ and dispersion $\sigma^2$, then
\[ n^{-1/2}\circ(X_1 + X_2 + \cdots + X_n) \eqd X \oplus (n^{1/2}-1)\mu \]
are both Hermite distributions with mean $n^{1/2}\mu$ and dispersion $\sigma^2$. This example is discussed in detail in Subsection \ref{sec:Hermite}.
\end{itemize}
\end{example}

\subsection{The rest of this paper}

The structure of the remainder of this paper is as follows.
\begin{itemize}
    \item In Section \ref{sec:back}, we give some useful background results about thinning, about the Poisson shift, which we argue is a natural discrete equivalent of the standard shift, and about the alternate probability generating function.
    \item In Section \ref{sec:dist}, we examine distributions related to discrete stability. Most importantly, we have the Poisson--delayed Sibuya distribution, which we show is discrete stable with $a_n = n^{-1/\alpha}$ for $\alpha \in (0, 2]$. We also discuss the special case of the Hermite distribution, which, by being discrete stable with stability parameter $\alpha = 2$, plays a similar role to the normal distribution, which is (usual non-discrete) stable with $\alpha = 2$.
    \item The main tricky technical work in this paper is completing the proof of Theorem \ref{th:dstab} by proving that $a_n$ \emph{must} be of the form $a_n = n^{-1/\alpha}$ for $\alpha \in (0,2]$ and that the \emph{only} discrete stable distributions are those with APGF as in \eqref{eq:maineq}; that is, the Poisson--delayed Sibuya distributions. This proof is in Section \ref{sec:proof} -- see the introduction to that section for an outline of the proof. 
\end{itemize}  


\section{Background results} \label{sec:back}

Before getting to our main proof, we state (and occasionally prove) some lemmas we will need later. In the process, we hope also to persuade the reader that the thinning and Poisson shift operations are the natural discrete equivalents of the scaling and shifting operations for real-valued random variables.


\subsection{Properties of thinning and the Poisson shift} \label{sec:shift}

Thinning was originally defined -- and is perhaps still most commonly studied -- in the context of point processes, following Rényi \cite{renyi}. When applied to a discrete random variable $X$, the thinning $a \circ X$ has the conditional binomial distribution $\Bin(X, a)$ given $X$. Work that has studied thinning as applied to discrete random variables includes \cite{steutel1,johnson,JK,kella,aldridge-thin}, for just a few of the many examples.

In what follows, we write
\[ \mathbb EX^{\underline{k}} = \mathbb EX(X-1)\cdots(X-k+1) \]
for the $k$th factorial moment of a random variable $X$. Following \cite{JK}, we write \[ \Disp(X) = \EE X^{\ff{2}} - (\EE X)^2 = \Var(X) - \EE X \]
for the \emph{dispersion} of $X$, which sometimes has nicer properties than the variance for discrete random variables.

Steutel and Van Harn \cite{steutel1} did not feel it necessary mount a particularly vigorous defence of the thinning operation as a natural discrete equivalent of scaling. However, had they done so, they might have mentioned these properties of thinning \dots

\begin{lemma} \label{lem:thin}
Let $X$ be a discrete random variable taking values in the non-negative integers, and let $a \in [0,1]$.
\begin{enumerate}
\item $\mathbb E(a \circ X) = a\,\mathbb EX$, if the expectation $\mathbb EX$ is finite.
\item $\Disp(a \circ X) = a^2\Disp(X)$ if the dispersion $\Disp(X)$ exists and is finite.
\item $\mathbb E(a \circ X)^{\underline{k}} = a^k \,\mathbb EX^{\underline{k}}$, if the $k$th factorial moment $\mathbb EX^{\underline{k}}$ is finite.
\item $a \circ (b \circ X) \eqd (ab) \circ X$ for $b \in [0,1]$.
\item $a \circ(X+Y) \eqd a\circ X + a\circ Y$ for independent $X$ and $Y$.
\item If $X$ is Poisson with mean $\mu$, then $a \circ X$ is Poisson with mean $a\mu$.
\end{enumerate}
\end{lemma}

\noindent\dots\ and compared them to these very similar properties of the standard scaling operation \dots

\begin{enumerate}
\item $\mathbb E(a X) = a\,\mathbb EX$, if the expectation $\mathbb EX$ exists and is finite.
\item $\Var(a X) = a^2\Var(X)$ if the variance $\Var(X)$ exists and is finite.
\item $\mathbb E(aX)^k = a^k \,\mathbb EX^k$, if the $k$th moment $\mathbb EX^k$ exists and is finite.
\item $a (b X) \eqd (ab) X$ for $b \in \mathbb R$.
\item $a (X+Y) \eqd aX + aY$ for independent $X$ and $Y$.
\item If $X$ is a point mass at $\mu$, then $aX$ is a point mass at $a\mu$.
\end{enumerate}

The similarity of these results show that thinning behaves in a similar way to scaling while keeping random variables within the non-negative integers. Thus thinning seems a natural discrete equivalent of scaling.

It is sometimes possible to extend the definition of thinning to allow $a > 1$ (see, for example, \cite{JK,kella}). We do not discuss that here, as it does not open up any extra discrete stable distributions.

The Poisson shift operation has appeared only sporadically in the literature, but some papers that do consider it are \cite{johnson,JK}. It can also sometimes be extended to `left shifts', where $b < 0$.

\begin{definition} \label{def:shift}
Let $X$ be a discrete random variable taking values in the non-negative integers, and let $b \geq 0$. The \emph{Poisson shift} $X \oplus b$ is defined as $X \oplus b = X + Z$, where $Z \sim \Po(b)$ is a Poisson distribution with mean $b$ that is independent of $X$.

If there exists a $Y$ such that $Y \oplus b \eqd X$, then we define this $Y$ to be the left Poisson shift of $X$, and write $Y = X  \oplus (-b)$ or $Y = X \ominus b$.
\end{definition}

To save on brackets, we use the convention that thinning takes precedence over Poisson shifting (just as multiplication takes precedence over addition), so $a \circ X \oplus b$ means $(a \circ X) \oplus b$.

The left Poisson shift $X \ominus b$ need not exist in general; but if we have an equation $X \eqd Y \oplus b$ for $b \geq 0$, we can always invert this to get $Y \eqd X \ominus b$. In particular, the definition of discrete stability \eqref{eq:dstrict}, from Definition \ref{def:dstrict}, can always be written in the form
\begin{equation} \label{eq:dstab2}
    a_n \circ (X_1 +X_2 + \cdots + X_n) \ominus b_n \eqd X .
\end{equation}
This is the form that we will find most convenient to use in what follows.

To argue in favour of the Poisson shift $X \oplus b$ as the natural discrete equivalent of the usual shift $X + b$ for real-valued random variables, we point out that these properties of the Poisson shift \dots

\begin{lemma} \label{lem:shift}
Let $X$ be a discrete random variable taking values in the non-negative integers. Provided the Poisson shifts in question exist, we have the following:
\begin{enumerate}
\item $\mathbb E(X \oplus b) = \mathbb EX + b$, if the expectation $\mathbb EX$ is finite.
\item $\Disp(X \oplus b) = \Disp(X)$ if the dispersion $\Disp(X)$ exists and is finite.
\item ${\displaystyle \mathbb E(X \oplus b)^{\underline{k}} = \sum_{j=0}^k \binom{k}{j} b^{k-j} \,\mathbb EX^{\underline{j}}}$, if the $k$th factorial moment $\mathbb EX^{\underline{k}}$ is finite.
\item $(X \oplus b) \oplus c \eqd X \oplus (b+c)$.
\item $(X + Y) \oplus b \eqd (X \oplus b) + Y \eqd X + (Y \oplus b)$.
\item $a \circ(X \oplus b) \eqd a \circ X + ab$ for $a \in [0,1]$.
\item If $X$ is Poisson with mean $\mu$, then $X \oplus b$ is Poisson with mean $\mu + b$.
\end{enumerate}
\end{lemma}

\noindent\dots\ seem to be natural discrete equivalents of these standard properties of the shift \dots
\begin{enumerate}
\item $\mathbb E(X + b) = \mathbb EX + b$, if the expectation $\mathbb EX$ exists and is finite.
\item $\Var(X + b) = \Var(X)$ if the variance $\Var(X)$ exists and is finite.
\item ${\displaystyle \mathbb E(X + b)^{k} = \sum_{j=0}^k \binom{k}{j} b^{k-j} \,\mathbb EX^{j}}$, if the $k$th moment $\mathbb EX^{k}$ exists.
\item $(X + b) + c \eqd X + (b+c)$.
\item $(X + Y) + b \eqd (X + b) + Y \eqd X + (Y + b)$.
\item $a (X + b) \eqd aX + ab$ for $a \in \mathbb R$.
\item If $X$ is point mass at $\mu$, then $X + b$ is a point mass at $\mu + b$.
\end{enumerate}

The proofs of Lemma \ref{lem:shift} are straightforward; part 2 uses the factorial binomial theorem (sometimes called Vandermonde's identity)
\[ (x + y)^\ff{k} = \sum_{j=0}^k \binom{k}{j} x^{\ff{j}} \,y^{\ff{k-j}} \]
and that the factorial moments of a Poisson distribution $Z \sim \Po(b)$ are $\EE X^{\ff j} = b^j$.

The similarity of these results show that the Poisson shift behaves in a similar way to the usual real-valued shift while keeping random variables within the non-negative integers. In particular, the distributive law in point 6 shows that thinning and the Poisson shift work well together. Thus Poisson shifting seems a natural discrete equivalent of shifting.

\subsection{Properties of the alternate probability generating function} \label{sec:apgf}

We remind the reader of the following properties of the alternate probability generating function (APGF) $\psi_X(t) = \mathbb E(1-t)^X$.

\begin{lemma} \label{lem:apgf}
Let $X$ be a discrete random variable taking values in the non-negative integers with alternate probability generating function $\psi_X$. Then we have the following properties:
\begin{align*}
\psi_{a\circ X}(t) &= \psi_X(at) \\
\psi_{X \oplus b}(t) &= \ee^{-bt}\, \psi_X(t) \\
\psi_{X + Y}(t) &= \psi_X(t)\,\psi_Y(t) , 
\end{align*}
where the third equation holds for independent random variables $X$ and $Y$. 
\end{lemma}
 
The first of these it what makes the APGF more convenient than the standard probability generating function, for which $G_{a \circ X}(t) = G_X(1 - a + at)$. The second follows since the APGF of a Poisson distribution with mean $b$ is $\ee^{-bt}$.

The three results of Lemma \ref{lem:apgf} can be though of as discrete equivalents of standard results about the characteristic function $\phi_X(t)$ of a general random variable $X$:
\begin{align*}
\phi_{aX}(t) &= \phi_X(at) \\
\phi_{X + b}(t) &= \ee^{\ii bt}\, \phi_X(t) \\
\phi_{X + Y}(t) &= \phi_X(t)\,\phi_Y(t) , 
\end{align*}

In particular, the APGF of
\[ Y = a \circ (X_1 + X_2 + \cdots + X_n) \ominus b , \]
for IID $X_1, X_2, \dots, X_n$ is $\psi_Y(t) = \ee^{bt} \, \psi_X(at)^n$. This is the equivalent of the characteristic function of
\[ Y = a (X_1 + X_2 + \cdots + X_n) - b , \]
being $\phi_Y(t) = \ee^{-\ii bt} \, \phi_X(at)^n$.

We write $C(t) = \log \psi_X(t)$ for the \emph{factorial cumulant generating function}. (We have taken the name from \cite{JK}, although they actually apply that name to $\log \psi_X(-u)$.) If the APGF is the exponential of a complicated expression, this can be more convenient to deal with.

\section{Discrete stable distributions}  \label{sec:dist}

The full range of real-valued stable distributions are difficult to describe directly, and have a rather complicated expression for their characteristic functions. In contrast, the discrete stable distributions can be described quite simply. This subsection will describe these distributions in more detail.
\begin{itemize}
  \item In Subsection \ref{sec:sib} we discuss the Sibuya distribution and its compound Poisson version, which we call the Poisson--Sibuya distribution. The Poisson--Sibuya distributions are discrete strictly stable.
  \item In Subsection \ref{sec:dsib} we define what we will call the `delayed Sibuya distribution' and its compound Poisson version, which we call the Poisson--delayed Sibuya distribution. We find the alternate probability generating function of the Poisson--delayed Sibuya distributions and show that they are discrete stable; this will require separate analysis for $\alpha \neq 1$ and $\alpha =1$.
  \item In Subsection \ref{sec:Hermite}, we discuss the Hermite distribution, which is discrete stable with $\alpha = 2$ but is not discrete strictly stable. We will compare its properties to that of the normal distribution, which is (usual non-discrete) stable with $\alpha = 2$.
\end{itemize}

\subsection{Sibuya and Poisson--Sibuya distributions} \label{sec:sib}

As explained by Devroye \cite{devroye}, the discrete strictly stable distributions are compound Poisson distributions where the compounded distribution is what he calls the `Sibuya distribution' (crediting \cite{sibuya}).

Consider a sequence of independent trials. The first trial succeeds with probability $\alpha$, the second trial succeeds with probability $\alpha/2$, the third with probability $\alpha/3$, and so on. Then the distribution of the number of the first successful trial is called the \emph{Sibuya distribution} with parameter $\alpha \in (0, 1]$. If $Y \sim \Sib(\alpha)$ is such a Sibuya distribution, then its probability mass function is $p(1) = \alpha$, $p(2) = (1-\alpha)\frac{\alpha}{2}$, and generally
\begin{align*}
p(y) &= (1 - \alpha) \Big(1 - \frac{\alpha}{2}\Big) \cdots \Big(1 - \frac{\alpha}{y-1}\Big) \frac{\alpha}{y} \\
  &= \frac{1}{y!} (1 - \alpha)(2-\alpha)\cdots\big((y-1)-\alpha\big)\,\alpha \\
  &= (-1)^{y+1} \binom{\alpha}{y} ,
\end{align*}
where we are extending the definition of the binomial coefficient
\[ \binom{\nu}{k} = \frac{\nu^\ff{k}}{k!} = 
\frac{\nu(\nu-1)\cdots(\nu - k + 1)}{k!} \]
to any real $\nu$. We must have $\alpha > 0$ to ensure there is a successful trial, and we must have $\alpha \leq 1$ to ensure that $p(1) = \alpha$ is a probability.

The APGF of the Sibuya distribution is
\begin{align*}
\psi_Y(t) &= \sum_{y=1}^\infty (-1)^{y+1} \binom{\alpha}{y} (1-t)^y \\ 
  &= 1 - \sum_{y=0}^\infty \binom{\alpha}{y} (-1+t)^y  \\
  &= 1 - \big(1 + (-1+t)\big)^\alpha \\
  &= 1 - t^\alpha .
\end{align*}
for $t \in [0,2]$.

Now let $X$ be the compound Poisson distribution with compounded distribution the Sibuya distribution, so $X = Y_1 + Y_2 + \cdots + Y_N$ where $N \sim \Po(\gamma)$ is a Poisson random number of summands and $Y_1, Y_2, \dots \sim \Sib(\alpha)$ are IID Sibuya distributions. We shall call this the \emph{Poisson--Sibuya distribution} and write $X \sim \PoSib(\gamma,\alpha)$.

The APGF of a compound Poisson distribution is $\psi_X(t) = \exp(\gamma(\psi_Y(t) - 1))$. It follows that the APGF of $X \sim \PoSib(\gamma, \alpha)$ is
\[ \psi_X(t) = \exp\big(\gamma(1 - t^\alpha - 1)\big) = \exp(-\gamma t^\alpha) . \]
By Theorem \ref{th:dstrict}, this is discrete strictly stable.


In the special case $\alpha = 1$, we see that $\Sib(1) = \delta(1)$ is simply a point mass at $1$, so the Poisson--Sibuya distribution $\PoSib(\gamma, 1) = \Po(\gamma)$ is just a Poisson distribution.


\subsection{Delayed Sibuya and Poisson--delayed Sibuya distributions} \label{sec:dsib}

We now define what we shall call the \emph{delayed Sibuya distribution}. It is almost the same as the usual Sibuya distribution, in that it is the number of the first success in a series of trials succeeding with probabilities $\theta, \alpha/2, \alpha/3, \dots$. The difference here is that the probability the first trial succeeds is no longer $\alpha$, but rather some different value $\theta$. That is, the Sibuya-like behaviour is `delayed' until the second trial. 
The probability mass function of $Y \sim \DSib(\theta, \alpha)$ is $p(1) = \theta$, $p(2) = (1-\theta) \frac{\alpha}{2}$, and, for $y \geq 2$, 
\begin{align*}
p(y) &= (1 - \theta) \Big(1 - \frac{\alpha}{2}\Big) \cdots \Big(1 - \frac{\alpha}{y-1}\Big) \frac{\alpha}{y} \\
  &= \frac{1}{y!} (1 - \theta)(2-\alpha) \cdots\big((y-1)-\alpha\big)\,\alpha \\
  &= (-1)^{y+1} \frac{1-\theta}{1-\alpha} \binom{\alpha}{y} ,
\end{align*}
with the final simplification only holding for $\alpha \neq 1$.

The delayed Sibuya is defined for $\theta \in [0,1]$ and (provided $\theta \neq 1$) $\alpha \in (0,2]$. The first condition $\theta \in [0,1]$ is required for $p(1) = \theta$ to be a probability. For $\theta \in [0, 1)$, we require $\alpha > 0$ to ensure there is almost surely a successful trial, and $\alpha \leq 2$ to ensure $p(3) = (1-\theta)(1 - \frac{\alpha}{2})\frac{\alpha}{3}$ is non-negative.

For $\theta \in [0,1]$ and $\alpha \in (0, 2]$ we write $Y \sim \DSib(\theta, \alpha)$ for such a delayed Sibuya distribution. Note the wider range of $\alpha$ now permitted: the delayed Sibuya distribution is defined for $\alpha \in (0,2]$ where the Sibuya distribution was only defined for $\alpha \in (0,1]$.

If $\alpha \leq 1$, then $\DSib(\alpha, \alpha) = \Sib(\alpha)$ recovers the normal Sibuya distribution. If $\alpha = 2$, then $\DSib(\theta, 2)$ is either $1$, with probability $\theta$, or $2$, with probability $1 - \theta$. 

The \emph{Poisson--delayed Sibuya distribution} $X \sim \PoDSib(\lambda,\theta,\alpha)$ is a compound Poisson distribution with $\Po(\lambda)$ many summands which are IID $\DSib(\theta, \alpha)$ delayed Sibuya distributions.

We now look at the APGF of the delayed Sibuya and Poisson--delayed Sibuya distributions, treating the cases $\alpha \neq 1$ and $\alpha = 1$ separately.

\subsubsection*{The case $\alpha \neq 1$}

We stick to $\alpha \in (0,1)\cup(1,2]$ for the moment; the $\alpha = 1$ case requires separate analysis, which follows shortly.

The APGF of the delayed Sibuya distribution $Y \sim \DSib(\theta, \alpha)$ is
\begin{align*}
\psi_X(t) &= \theta(1-t) + \sum_{y=2}^\infty (-1)^{y+1} \frac{1-\theta}{1-\alpha} \binom{\alpha}{y} (1-t)^y  \\
  &= \theta(1-t) + \frac{1-\theta}{1-\alpha}\left(1 - \alpha(1-t) + \sum_{y=0}^\infty \binom{\alpha}{y} (-1+t)^y\right) \\
  &= 1 - \frac{\theta - \alpha}{1 - \alpha} t - \frac{1-\theta}{1-\alpha} \, t^\alpha \\
  &= 1 - (1 - \zeta)t - \zeta t^\alpha ,
\end{align*}
where $\zeta = \frac{1-\theta}{1-\alpha}$. Note that $\zeta$ satisfies $\zeta \geq 0$ when $\alpha < 1$ and $\zeta \leq 0$ when $\alpha > 1$.

The Poisson--delayed Sibuya distribution $X \sim \PoDSib(\lambda, \theta \alpha)$ therefore has APGF
\[ \psi_X(t) = \exp\big(\lambda\big({-(1 - \zeta)t} - \zeta t^\alpha\big)\big) = \exp(-\delta t - \gamma t^\alpha) , \]
where 
\[ \delta = \lambda(1 - \zeta) = \lambda \,\frac{\theta - \alpha}{1 - \alpha} \qquad \gamma = \lambda\zeta = \lambda \,\frac{1-\theta}{1-\alpha} . \]
Solving for $\lambda$ and $\theta$ in terms of $\delta$ and $\gamma$ gives
\[ \lambda = \delta + \gamma \qquad \theta = \frac{\delta + \alpha\gamma}{\delta + \gamma} . \]
We know $\lambda$ must be non-negative. So for $\alpha < 1$, $\zeta$ is positive, giving $\gamma \geq 0$, while  for $\alpha > 1$, $\zeta$ is negative, so $\gamma \leq 0$.
The condition $\lambda \geq 0$ means $\delta + \gamma > 0$; hence the condition $\theta \in [0,1]$ corresponds to $\delta \geq -\alpha\gamma$. These are the conditions in Theorem \ref{th:dstab}.

We can now prove the Poisson--delayed Sibuya distributions with $\alpha \neq 1$ are discrete stable.

\begin{theorem} \label{th:DSib}
Let $X \sim \PoDSib(\lambda, \theta, \alpha)$ be a Poisson--delayed Sibuya distribution with $\alpha \in (0,1) \cup (1,2]$ and $\theta \in [0, 1)$. Then $X$ is discrete stable, in the sense of Definition \ref{def:dstab}, with $a_n = n^{-1/\alpha}$.
\end{theorem}

\begin{proof}
Set $a_n = n^{-1/\alpha}$ and $b_n = \delta(n^{1 - 1/\alpha} - 1)$. Then $a_n \circ(X_1 + X_2 + \cdots + X_n)$ has APGF
\begin{align*}
\psi_X(n^{-1/\alpha}t)^n
  &= \big(\exp(-\delta n^{-1/\alpha}t - \gamma n^{-1} t^\alpha)\big)^n 
  = \exp( -\delta n^{1-1/\alpha}t - \gamma t^\alpha)  ,
\end{align*}
while $X \oplus b_n$ has APGF
\[ \ee^{-b_n t} \psi_X(t) = \exp({- \delta(n^{1 - 1/\alpha} - 1)} t - \delta t - \gamma t^\alpha) , \]
which is the same. Hence Definition \ref{def:dstab} is satisfied and $X$ is discrete stable.
\end{proof}

\subsubsection*{The case $\alpha = 1$}

We return to the case $\alpha = 1$. Here, $Y \sim \DSib(\theta, 1)$ has probability mass function $p(1) = \theta$ and, for $y \geq 2$,
\begin{align*}
p(y) &= (1-\theta)\bigg(1 - \frac12\bigg)\bigg(1 - \frac13\bigg) \cdots \bigg(1 - \frac{1}{y-1}\bigg)\frac{1}{y} \\
  &= (1 - \theta) \,\frac12\cdot\frac23\cdots\frac{y-2}{y-1}\cdot\frac{1}{y} \\
  &= (1-\theta) \,\frac{1}{(y-1)y} \\
  &= (1 - \theta) \bigg(\frac{1}{y-1} - \frac{1}{y}\bigg) .
\end{align*}

Therefore then APGF of $Y$ is
\begin{align*}
\psi_Y(t) &= \theta(1-t) + \sum_{y=2}^\infty (1 - \theta) \bigg(\frac{1}{y-1} - \frac{1}{y}\bigg) (1-t)^y \\
  &= \theta (1-t) + (1-\theta) \left((1-t) + \sum_{y=1}^\infty \frac{1}{y} \big(-(1-t)^y + (1-t)^{y+1}\big) \right) \\
  &= (1-t) + (1-\theta) \left({-t}\sum_{y=1}^\infty \frac{1}{y} (1-t)^y \right) \\
  &= 1 - t + (1-\theta) t \log t ,
\end{align*}
for $t \in [0,2]$ (interpreting $0 \log 0 = 0$).

The Poisson--delayed Sibuya distribution $X \sim \PoDSib(\lambda, \theta, 1)$ therefore has APGF
\[ \psi_X(t) = \exp \big({\lambda}\big({-t} + (1-\theta)t \log t \big)\big) = \exp(-\delta t - \gamma t \log t ) , \]
where, $\delta = \lambda$ and $\gamma = -\lambda(1-\theta)$. The conditions $\lambda \geq 0$ and $\theta \in [0,1]$ correspond to $\delta \geq 0$ and $\gamma \in [-\delta, 0]$.

We can now prove the Poisson--delayed Sibuya distributions with $\alpha = 1$ are discrete stable too.

\begin{theorem}
Let $X \sim \PoDSib(\lambda, \theta, 1)$ be a Poisson--delayed Sibuya distribution with $\theta \neq 1$. Then $X$ is discrete stable, in the sense of Definition \ref{def:dstab}, with $a_n = n^{-1}$.
\end{theorem}

\begin{proof}
Set $a_n = n^{-1}$ and $b_n = \gamma \log n$. Then $a_n \circ(X_1 + X_2 + \cdots + X_n)$ has APGF
\begin{align*}
\psi_X(n^{-1}t)^n
  &= \big(\exp\big ({-\delta} n^{-1} t - \gamma n^{-1} t \log (n^{-1}t) \big)\big)^n \\
  &= \exp\big( {-(\delta + \gamma n\log n)}t - \gamma t\log t\big) ,
\end{align*}
while $X \oplus b_n$ has APGF
\[ \ee^{-b_n t} \psi_X(t) = \exp({-\gamma} t \log n - \delta t - \gamma t \log t) , \]
which is the same. Hence Definition \ref{def:dstab} is satisfied.
\end{proof}

\subsection{Hermite distribution} \label{sec:Hermite}

The most important special case of a discrete stable distribution is the Hermite distribution \cite{kemp1}. (We use the parametrisation later used in \cite{kemp2,JK}.)

\begin{definition} \label{def:herm}
A random variable $X$ is said to have the \emph{Hermite distribution} with mean $\mu$ and dispersion $\sigma^2$, where $\mu \geq \sigma^2 \geq 0$, if $X \eqd U + 2V$, where $U \sim \Po(\mu - \sigma^2)$ and $V \sim \Po(\sigma^2/2)$ are independent Poisson distributions. We write $X \sim \Herm(\mu, \sigma^2)$.

Another equivalent definition is the following. Let $X$ be a compound Poisson distribution $X = W_1 + W_2 + \cdots W_N$, where $N \sim \Po(\lambda)$ and $W_1, W_2, \dots$ are IID with distribution
\[ W_i = \begin{cases} 
1 & \text{with probability $\theta$} \\
2 & \text{with probability $1 - \theta$.} \end{cases} \]
Then $X \sim \Herm(\mu, \sigma^2)$, where $\mu = \lambda(\theta + 2(1 - \theta)) = \lambda (2 - \theta)$ and $\sigma^2 = 2\lambda(1-\theta)$. 
\end{definition}

Importantly for us, this $W$ is a delayed Sibuya distribution $\DSib(\theta, 2)$ with $\alpha = 2$, and hence the Hermite distribution is a Poisson--delayed Sibuya distribution. Indeed, the Hermite distributions are the only discrete stable distributions with finite variance $\operatorname{Var}(X) = \mu + \sigma^2$ or dispersion $\Disp(X) = \sigma^2$.

We note the following facts about the Hermite distribution.

\begin{lemma} \label{lem:herm}
\mbox{}
\begin{enumerate}
\item The alternate probability generating function of a Hermite distribution $X \sim \Herm(\mu,\sigma^2)$ is \[ \psi_X(t) = \exp\big({-\mu t} + \tfrac12\sigma^2 t^2\big) . \]
\item A sum of independent Hermite distributions is Hermite, in that
\[ \Herm(\mu_1, \sigma_1^2) + \Herm(\mu_2, \sigma_2^2) = \Herm(\mu_1 + \mu_2, \sigma_1^2 + \sigma_2^2) . \]
\item A thinning of a Hermite distribution is still Hermite, in that, for $a \in [0,1]$,
\[ a \circ \Herm(\mu, \sigma^2) = \Herm(a\mu, a^2\sigma^2). \]
\item A Poisson shift of a Hermite distribution is still Hermite, in that, for $b \geq -(\mu - \sigma^2)$,
\[ \Herm(\mu, \sigma^2) \oplus b = \Herm(\mu + b, \sigma^2) .\]
\item A Hermite distribution is not discrete strictly stable, unless $\sigma^2 = 0$.
\item A Hermite distribution is discrete stable with $a_n = n^{-1/2}$.
\end{enumerate}
\end{lemma}

\begin{proof}
Part 1 follows by using Definition \ref{def:herm} and either $\psi_X(t) = \psi_U(t)\,\psi_V(2t - t^2)$ or $\psi_X(t) = \exp(\lambda(\psi_W(t) - 1))$. Parts 2 and 4 follow immediately from Definition \ref{def:herm}. Part 3 follows from Part 1 and $\psi_{a \circ X}(t) = \psi_X(at)$ (Lemma \ref{lem:apgf}). Part 5 is because the APGF is not of the form in Theorem \ref{th:dstrict}. Part 6 can be shown directly, as in Example \ref{ex:examples}, or as a consequence of the general result for Poisson--delayed Sibuya distributions (Theorem \ref{th:DSib}).
\end{proof}

The results of Lemma \ref{lem:herm} can be seen as the discrete equivalents of well known results for the normal distribution:
\begin{enumerate}
\item The characteristic function of a normal distribution $X \sim \NN(\mu,\sigma^2)$ is \[ \phi_X(t) = \exp\big(\ii\mu t - \tfrac12\sigma^2 t^2\big) . \]
\item A sum of independent normal distributions is normal, in that
\[ \NN(\mu_1, \sigma_1^2) + \NN(\mu_2, \sigma_2^2) = \NN(\mu_1 + \mu_2, \sigma_1^2 + \sigma_2^2) . \]
\item A scaling of a normal distribution is still normal, in that, for $a \in \mathbb{R}$,
\[ a \, \NN(\mu, \sigma^2) = \NN(a\mu, a^2\sigma^2). \]
\item A shift of a normal distribution is still normal, in that, for $b \in \mathbb{R}$,
\[ \NN(\mu, \sigma^2) + b = \NN(\mu + b, \sigma^2) .\]
\item A normal distribution is not strictly stable, unless $\mu = 0$ or $\sigma^2 = 0$.
\item A normal distribution is stable with $a_n = n^{-1/2}$.
\end{enumerate}




\section{Proof of main theorem} \label{sec:proof}

We have proved that the Poisson--delayed Sibuya distributions have APGFs as in \eqref{eq:maineq} and are discrete stable with $a_n = n^{-1/\alpha}$ for $\alpha \in (0,2]$. The remaining technical part of our main result, Theorem \ref{th:dstab}, is to show that $a_n$ \emph{must} be of the form $a_n = n^{-1/\alpha}$ for $\alpha \in (0, 2]$, and that the APGF \emph{must} be that of the Poisson--delayed Sibuya distribution as in \eqref{eq:maineq}.
\begin{itemize}
\item In Subsection \ref{sec:lemmas}, we prove two lemmas that will be useful for the proof. These are not vital for understanding the `big idea' of how the proof works
\item In Subsection \ref{sec:an}, we prove that $(a_n)$ satisfies $a_{nm} = a_n a_m$ and hence that $a_n = n^{-1/\alpha}$.
\item In Subsection \ref{sec:bn}, we find formulas for $b_n$.
\item In Subsection \ref{sec:phi}, we show that the APGF must be of the form in \eqref{eq:maineq}.
\end{itemize}
Many parts of the proof will have to treat the cases $\alpha \neq 1$ and $\alpha = 1$ separately.

\subsection{Two useful lemmas} \label{sec:lemmas}

In this subsection, we prove two useful lemmas. Readers wanting to get the main ideas of our proof can skip straight to the next section, however.
\begin{itemize}
    \item In Lemma \ref{lem:class}, we show that discrete stable distributions are necessarily compound Poisson.
    \item Lemma \ref{lem:equate} is a useful result that allows one to `equate' thinning probabilities and shift parameters in equations involving thinning and the Poisson shift
\end{itemize}

The discrete stable distributions we have seen, the Poisson--Sibuya distributions, are compound Poisson distributions. We can prove directly from the definition of discrete stability (Definition \ref{def:dstab}) that any discrete stable distribution must be compound Poisson. This will make life easier for us later in the proof.


\begin{lemma} \label{lem:class}
Let $X$ be a discrete stable distribution. Then we have the following.
\begin{enumerate}
    \item $X$ is infinitely divisible.
    \item $X$ is a compound Poisson distribution.
\end{enumerate}
\end{lemma}

\begin{proof}
For the first part, we have from \eqref{eq:dstab2} that
\[ X = a_n \circ (X_1 + \cdots + X_n) \ominus b_n = \big(a_n \circ X_1 \ominus \tfrac{b_n}{n} \big) + \cdots + \big(a_n \circ X_n \ominus \tfrac{b_n}{n} \big) , \]
where the rearrangement has used properties of thinning and the Poisson shift from Lemmas \ref{lem:thin} and \ref{lem:shift}. This shows that $X$ is infinitely divisible, in that, for any $n$, we can write $X$ as a sum of $n$ IID copies of $Y_n = a_n \circ X \ominus \tfrac{b_n}{n}$.

The second part follows from a result of Feller \cite[Section XII.2]{feller1} (following the more general results of Lévy), which states that a random variable $X$ on the non-negative integers is infinitely divisible with $Y_n$ also taking values in the non-negative integers if and only if $X$ is a compound Poisson distribution.
\end{proof}

The following lemma allows us to `equate thinning probabilities' and `equate shift parameters'; it will be important later in the proof of our main result.

\begin{lemma} \label{lem:equate}
Consider a random variable $X$ taking values in the non-negative integers that is not a Poisson distribution. Suppose $a\circ X \oplus b \eqd c \circ X \oplus d$ for some $a, c \in [0,1]$ and $b, d \in \mathbb R$, where both Poisson shifts exist. Then $a = c$ and $b = d$.
\end{lemma}

\begin{proof}
Without loss of generality, we may assume that $c = 1$ (arrange things so $a \leq c$; replace $X$ with $c \circ X$ and $a$ with $a/c$) and that $d = 0$ (replace $b$ with $b - d$). Thus it suffices show that, if $a \circ X \oplus b \eqd X$ and $X$ is not Poisson, then $a = 1$ and $b = 0$.

By properties of the APGF (Lemma \ref{lem:apgf}) and the definition of the factorial cumulant generating function $C$, the statement $a \circ X \oplus b \eqd X$  means that $C(au) - bu = C(u)$ for all $u \in [0,2]$. Setting $u = v, av, \dots, a^kv$ for $v \in [0, 2]$ gives
\begin{align*}
C(v) &= C(av) - bv \\
C(av) &= C(a^2v) - bav \\
&\;\, \vdots \\
C(a^{k}v) &= C(a^{k+1}v) - ba^{k}v .
\end{align*}
Substituting each equation into the one above, we get
\begin{align*} 
C(v) &= C(av) - bv \\
  &= C(a^2v) - b(1 + a)v \\
  &= \,\cdots \\
  &= C(a^{k+1}v) - b(1 + a + \cdots + a^k)v
\end{align*}
If it were the case that $a \neq 1$, then in the limit as $k \to \infty$, we would have
\[ C(v) = C(0) - b(1 + a + a^2 + \cdots)v = 0 - \frac{b}{1-a}v , \]
by continuity of $C$.
So either $C(u) = -\frac{b}{1-a}u$ for all $u \in [0,2]$, in which case $X$ is Poisson with rate $\frac{b}{1-a}$, which was forbidden in the statement of the lemma, or else $a = 1$.

Given $a = 1$, then $C(u) = C(u) - bu$, so $b = 0$ also.
\end{proof}

\subsection{Showing that $a_n = n^{-1/\alpha}$} \label{sec:an}

In this subsection, we prove the first important part of our result.
\begin{itemize}
    \item In Lemma \ref{lem:anrec}, we show that $a_n$ satisfies $a_{nm} = a_na_m$. This is perhaps the most vital step of the proof.
    \item In Lemma \ref{lem:anform}, we show that this means that $a_n = n^{-1/\alpha}$. This is intuitive from Lemma \ref{lem:anrec}, but checking formally that it holds is slightly technical.
\end{itemize}
The proof that $\alpha$ must be in $(0, 2]$ will come later, in Subsection \ref{sec:phi}.

\begin{lemma} \label{lem:anrec}
Suppose that $X$ is discrete stable and that $X$ is not a Poisson distribution. Then $a_{nm} = a_na_m$ for all $n, m$.
\end{lemma}

\begin{proof}
We apply the definition of discrete stability in two ways. Applying the definition \eqref{eq:dstab2} for $nm$ variables gives
\begin{equation} \label{eq:apply1}
X \eqd a_{nm} \circ (X_1 + \cdots + X_{nm}) \ominus b_{nm} .
\end{equation}
Applying the definition \eqref{eq:dstab2} with $n$ variables (once) then with $m$ variables ($n$ times) gives
\begin{align}
X & \eqd a_n \circ (X_1 + \cdots + X_n) \ominus b_n \notag \\
  & \eqd a_n \circ \big(a_m \circ (X_1 + \cdots + X_m) \ominus b_m + \cdots \notag\\
  &\qquad\qquad\  {}+ a_m \circ(X_{(n-1)m+1} + \cdots + X_{nm}) \ominus b_m\big) \ominus b_n \notag\\
  & \eqd (a_n a_m) \circ (X_1 + \cdots + X_{nm}) \ominus (na_nb_m + b_n) ,\label{eq:apply2}  
\end{align}
where we have used Lemmas \ref{lem:thin} and \ref{lem:shift} to simplify the expression.
Equating \eqref{eq:apply1} and \eqref{eq:apply2} gives
\begin{equation} \label{eq:equate}
a_{nm} \circ (X_1 + \cdots + X_{nm}) \ominus b_{nm} \eqd (a_n a_m) \circ (X_1 + \cdots + X_{nm}) \ominus (na_nb_m + b_n) .
\end{equation}

Since $X$ is not Poisson, we can use Lemma \ref{lem:equate} to equate thinning probabilities in \eqref{eq:equate}. Equating thinning probabilities gives $a_{nm} = a_n a_m$.
\end{proof}

The equation $a_{nm} = a_n a_m$ is highly suggestive that $a_n$ is of the form $n$ to a fixed power. We just have to prove this formally.

\begin{lemma} \label{lem:anform}
Suppose that $X$ is discrete stable and that $X$ is not a Poisson distribution. Then $a_{n} = n^{-1/\alpha}$.   
\end{lemma}

\begin{proof}
Lemma \ref{lem:anrec} shows that $(a_n)$ is completely multiplicative. If we can show that $(a_n)$ is also decreasing, then that will prove that $a_n = n^\beta$ (see, for example, \cite{howe}) for some $\beta < 0$ , and we can set $\beta = -1/\alpha$.

All that remains is to show $(a_n)$ is decreasing. We consider the factorial cumulant generating function $C_X(t) = \log \psi_X(t)$. From \eqref{eq:dstab2}, we have $C_X(t) = b_n t + n\,C_X(a_nt)$. Differentiating twice and setting $t = 1$ gives
\begin{equation} \label{eq:andec}
C_X''(1) = na_n^2 C_X''(a_n) .
\end{equation}
Since $X$ is compound Poisson (from Lemma \ref{lem:class}), $C''_X$ is increasing, and since $X$ is not Poisson, $C''$ is not identically $0$. So in \eqref{eq:andec}, the $C''_X(1)$ on the left-hand side is constant , the $n$ on the right-hand side increasing and $a_n^2 C''(a_n)$ must be decreasing, which requires $a_n$ to be decreasing too.
\end{proof}

\subsection{Formulas for $b_n$} \label{sec:bn}

We have now shown that $a_n = n^{-1/\alpha}$. In this subsection, we prove formulas for $b_n$, treating the cases $\alpha \neq 1$ (Lemma \ref{lem:bnaneq}) and $\alpha = 1$ (Lemma \ref{lem:bna1}) separately.

\begin{lemma} \label{lem:bnaneq}
Suppose that $X$ is discrete stable, that $X$ is not a Poisson distribution, and that $a_n = n^{-1/\alpha}$ for $\alpha \neq 1$. Then $b_n = \delta(n^{1 - 1/\alpha} - 1)$ for some $\delta$.
\end{lemma}

\begin{proof}
We start by going back to \eqref{eq:equate}, which, with $a_n = n^{1/\alpha}$, was
\begin{multline*}
(nm)^{-1/\alpha} \circ (X_1 + \cdots + X_{nm}) \ominus b_{nm} \\ \eqd n^{-1/\alpha}m^{-1/\alpha} \circ (X_1 + \cdots + X_{nm}) \ominus (n\,n^{-1/\alpha}b_m + b_n) .
\end{multline*}
Hence, $b_{nm} = n^{1-1/\alpha}b_m + b_n$. But also, swapping the roles of $n$ and $m$, $b_{nm} = b_{mn} = m^{1-1/\alpha}b_n + b_m$. Hence, for all $m$ and $n$, we have
\[ n^{1-1/\alpha}b_m + b_n = m^{1-1/\alpha}b_n + b_m , \]
and, for $\alpha \neq 1$, we can rearrange this to
\[ \frac{b_n}{n^{1-1/\alpha} - 1} = \frac{b_m}{m^{1-1/\alpha} - 1} . \]
Call this common value $\delta$. Then $b_n = \delta(n^{1 - 1/\alpha} - 1)$ for all $n$.
\end{proof}

\begin{lemma} \label{lem:bna1}
Suppose that $X$ is discrete stable, that $X$ is not a Poisson distribution, and that $a_n = n^{-1}$. Then:
\begin{enumerate}
    \item $b_{nm} = b_n + b_m$ for all $n, m$;
    \item $b_n = \gamma \log n$ for some $\gamma$.
\end{enumerate}
\end{lemma}

\begin{proof}
Part 1 is the same as the start of the proof for the $\alpha \neq 1$ case. We go back to $\eqref{eq:equate}$, set $a_n = n^{-1}$, and get
\[ (nm)^{-1} \circ (X_1 + \cdots + X_{nm}) \ominus b_{nm} \\ \eqd n^{-1}m^{-1} \circ (X_1 + \cdots + X_{nm}) \ominus (n\,n^{-1}b_m + b_n) , \]
so $b_{nm} = b_m + b_n$.

That $(b_n)$ is additive strongly suggests that $b_n = \gamma \log n$. To prove it, as with the proof of Lemma \ref{lem:anform}, we just need to show that $(b_n)$ is decreasing.

Using the factorial cumulant generating function again, we have \[ b_nt + n\,C_X(n^{-1}t) = C_X(t) . \] Setting $t = 1$, and noting that $C_X(1) = \log 1 = 0$, we have $b_n = -n\,C_X(n^{-1})$. We claim that the function $g(u) = -u \,C_X(u^{-1})$ is decreasing in $u$, and hence that $b_n$ is decreasing in $n$. To prove this we can show that the derivative $g'$ is negative. The derivative is
\[ g'(u) = -C_X(u^{-1}) + u^{-1} \,C'_X(u^{-1}) = -C_X(v) + v \, C'_X(v) ,  \]
setting $v = u^{-1}$.
But, since $X$ is compound Poisson, $C_X$ is concave. Hence
\[ C_X(v) \geq C_X(0) + v\, C_X'(v) = v\,C_X'(v) ,\]
so $g'(u)$ is negative, and we are done.
\end{proof}

\subsection{Form of the alternate probability generating function} \label{sec:phi}

We can now prove the form of the APGF of a discrete stable distribution is as in Theorem \ref{th:dstab}, thereby completing our proof.
\begin{itemize}
    \item For the $\alpha \neq 1$ case, Lemma \ref{lem:fin1} proves the general form and Lemma \ref{lem:fin2} the conditions on $\delta$ and $\gamma$.
    \item For the $\alpha = 1$ case, Lemma \ref{lem:fin3} proves the general form and Lemma \ref{lem:fin4} the conditions on $\delta$ and $\gamma$.
\end{itemize}

\begin{lemma} \label{lem:fin1}
Suppose that $X$ is discrete stable and that $a_n = n^{-1/\alpha}$ for $\alpha \neq 1$. Then the factorial cumulant generating function of $X$ is $C_X(t) = -\delta t - \gamma t^\alpha$ for some $\delta$ and $\gamma$.
\end{lemma}

\begin{proof}
We have from Lemma \ref{lem:bnaneq} that $b_n = \delta(n^{1 - 1/\alpha} - 1)$. So
\[ C_X(t) = \delta(n^{1 - 1/\alpha} - 1)t + n\,C_X(n^{-1/\alpha} t) . \]
Set $t = 1$ and rearrange, to get
\begin{align*}
C_X(n^{-1/\alpha}) &= n^{-1} \big({-\delta}(n^{1 - 1/\alpha} - 1)t + C_X(1)\big) \\
  &= -\delta n^{-1/\alpha} - \big(\delta - C_X(1)\big)n^{-1} \\
  &= -\delta n^{-1/\alpha} - \gamma n^{-1} ,
\end{align*}
where we have written $\gamma = \delta - C_X(1)$. If $t$ is of the form $t = n^{-1/\alpha}$, then $n^{-1} = t^{\alpha}$, and we have
\[ C_X(t) = \delta t - \gamma t^\alpha , \]
for those $t$.

This isn't quite enough -- but it will be sufficient if we can show this holds for all $t = r^{-1/\alpha}$ for rational $r$, since these are dense in $[0,1]$. To do this, we need to check that we also have $C_X(t) = b_r t + r\,C_X(a_r t)$ where $a_r = r^{-1/\alpha}$ and $b_r = \delta(r^{1 - 1/\alpha} - 1)$ for any rational $r$. For any $n, m$ such that $r = n/m$, we have
\[ C_X(t) = b_n t + n\, C_X(a_n t) = b_m t + m\, C_X(a_m t) . \]
Setting $t = \frac{1}{a_m}\,s$, we have
\[ \frac{b_n}{a_m}\,s + n \, C_X \Big( \frac{a_n}{a_m}\,s\Big) = \frac{b_m}{a_m} \, s + m\,C_X(s) , \]
and rearranging gives
\[ C_X(s) = \frac{b_n - b_m}{ma_m}\, s + \frac{n}{m}\,C_X\Big( \frac{a_n}{a_m}\,s\Big) .\]
But
\[ \frac{a_n}{a_m} = \frac{n^{-1/\alpha}}{m^{-1/\alpha}} = r^{-1/\alpha} \qquad
\frac{b_n - b_m}{ma_m} = \delta\,\frac{n^{1-1/\alpha} - m^{1 - 1/\alpha}}{m^{1-1/\alpha}} = \delta(r^{1 - 1/\alpha} - 1) , \]
as required.
\end{proof}

\begin{lemma} \label{lem:fin3}
Suppose that $X$ is discrete stable and that $a_n = n^{-1}$. Then the factorial cumulant generating function of $X$ is $C_X(t) = \delta t - \gamma t \log t$ for some $\delta$ and $\gamma$.
\end{lemma}

\begin{proof}
We have from Lemma \ref{lem:bna1} that $b_n = -\gamma \log n$. So
\[ C_X(t) = -\gamma (\log n)t + n\,C_X(n^{-1} t) . \]
Set $t = 1$ and rearrange, to get
\begin{align*}
C_X(n^{-1}) &= n^{-1} \big({-\gamma} \log n + C_X(1)\big) \\
  &= n^{-1}C_X(1) + \gamma n^{-1} \log n  \\
  &= -\delta n^{-1} + \gamma n^{-1}\log n ,
\end{align*}
where we have written $\delta = -C_X(1)$. If $t$ is of the form $t = n^{-1}$, then $n^{-1} = t$ and $\log n = -\log t$, so we have
\[ C_X(t) = -\delta t - \gamma t \log t , \]
for those $t$.

This isn't quite enough -- but it will be sufficient if we can show this holds for all $t = r^{-1}$ for rational $r$, since these are dense in $[0,1]$. To do this, we need to check that we also have $C_X(t) = b_r t + r\,C_X(a_r t)$ where $a_r = r^{-1}$ and $b_r = -\gamma \log r$ for any positive rational $r$. For any $n, m$ such that $r = n/m$, we have
\[ C_X(t) = b_n t + n\, C_X(a_n t) = b_m t + m\, C_X(a_m t) . \]
As before, setting $t = \frac{1}{a_m}\,s$ gives
\[ C_X(s) = \frac{b_n - b_m}{ma_m}\, s + \frac{n}{m}\,C_X\Big( \frac{a_n}{a_m}\,s\Big) .\]
But
\[ \frac{a_n}{a_m} = \frac{n^{-1}}{m^{-1}} = r^{-1} \qquad
\frac{b_n - b_m}{ma_m} = \frac{{-\gamma}(\log n - \log m)}{1} = -\gamma \log r , \]
as required.
\end{proof}

We now need to check the conditions of $\delta$ and $\gamma$ from \eqref{eq:maineq} are required. We already know $(a_n)$ is decreasing (from the proof of Lemma \ref{lem:anform}), so $\alpha > 0$.

\begin{lemma} \label{lem:fin2}
Suppose that $X$ is a compound Poisson distribution with APGF $\psi_X(t) = \exp(-\delta t - \gamma t^\alpha)$ for $\alpha \neq 0, 1$ and $\gamma \neq 0$.
\begin{enumerate}
\item If $\alpha \in (0, 1)$, $\gamma \geq 0$ and $\delta \geq -\alpha\gamma$.
\item If $\alpha > 1$, then $\alpha \leq 2$, $\gamma \leq 0$ and $\delta \geq - \alpha \gamma$.
\end{enumerate}
\end{lemma}

\begin{proof}
Our life is made easier by using the compound Poisson form;  generically, this means
\[ \psi_Y(t) = 1 + \frac{1}{\lambda} \log \psi_X(t) . \]
The distribution $Y$ that gets compounded can be put in `canonical form' where $\mathbb P(Y = 0) = \psi_Y(1) = 0$, meaning that $\lambda = -\log \psi_X(1)$.

In our case, $Y$ has APGF 
\[ \psi_Y(t) = 1 + \frac{1}{\delta + \gamma} (-\delta t - \gamma t^\alpha) , \]
where $\delta + \gamma \geq 0$.

We note that
\[ \mathbb P(Y = 1) = -\psi_Y'(1) = -\frac{-\delta - \alpha\gamma}{\delta + \gamma} = \frac{\delta + \alpha\gamma}{\delta + \gamma}. \]
To get $\phi'_Y(1) \geq 0$, we need $\delta + \alpha\gamma \geq 0$, so $\delta \geq -\alpha\gamma$.
To get $\psi_Y'(1) \leq 1$, we need $\delta + \alpha\gamma \leq \delta + \gamma$, so $\alpha \gamma \leq \gamma$. For $\alpha < 1$, this requires $\gamma \geq 0$; for $\alpha > 1$, this requires $\gamma \leq 0$.

Similarly,
\[ \mathbb P(Y = 3) = -\psi_Y'''(1) = -\frac{-\gamma}{\delta + \gamma}\, \alpha(\alpha - 1)(\alpha - 2) . \]
For $\alpha > 1$, where $\gamma \leq 0$, having $-\psi_Y'''(1) \geq 0$ means $\alpha (\alpha - 1) (\alpha - 2)$ must be negative, which means we must have $\alpha - 2 \leq 0$ or $\alpha < 2$.
\end{proof}

\begin{lemma} \label{lem:fin4}
Suppose that $X$ is a compound Poisson distribution with APGF $\psi_X(t) = \exp(-\delta t - \gamma t\log t)$.
Then $\gamma \leq 0$ and $\delta \geq -\gamma$.
\end{lemma}

\begin{proof}
Again, we use the compound Poisson with $Y$ in canonical form, where
\[ \psi_Y(t) = 1 + \frac{1}{\delta} (-\delta t - \gamma t \log t)  ,\]
where $\delta \geq 0$.
Here,
\[ \mathbb P(Y = 1) = -\psi'_Y(1) = -\frac{-\delta - \gamma}{\delta} = \frac{\delta+\gamma}{\delta}. \]
Then $-\psi'_Y(1) \geq 0$ requires $\delta + \gamma \geq 0$, or $\delta \geq -\gamma$, and $-\psi'_Y(1) \leq 1$ requires $\delta + \gamma \leq \delta$, or $\gamma \leq 0$.
\end{proof}

\section*{Acknowledgements}

The author thanks Nadjib Bouzar, Sam Power, and Will Townes for useful discussions and information.

\bibliographystyle{abbrvurl}
\bibliography{bibliography}

\begin{thebibliography}{10}

\bibitem{aldridge-thin}
M.~Aldridge.
\newblock The law of thin processes: a law of large numbers for point processes, 2025.
\newblock \href {https://arxiv.org/abs/2502.14839} {\path{arXiv:2502.14839}}.

\bibitem{devroye}
L.~Devroye.
\newblock A triptych of discrete distributions related to the stable law.
\newblock {\em Statistics \& Probability Letters}, 18(5):349--351, 1993.
\newblock \href {https://doi.org/10.1016/0167-7152(93)90027-G} {\path{doi:10.1016/0167-7152(93)90027-G}}.

\bibitem{feller1}
W.~Feller.
\newblock {\em An Introduction to Probability Theory and Its Applications}, volume~I.
\newblock John Wiley \& Sons, third edition, 1968.

\bibitem{feller2}
W.~Feller.
\newblock {\em An Introduction to Probability Theory and Its Applications}, volume~II.
\newblock John Wiley \& Sons, second edition, 1971.

\bibitem{howe}
E.~Howe.
\newblock A new proof of {E}rd{\H o}s's theorem on monotone multiplicative functions.
\newblock {\em The American Mathematical Monthly}, 93(8):593--595, 1986.
\newblock \href {https://doi.org/10.1080/00029890.1986.11971896} {\path{doi:10.1080/00029890.1986.11971896}}.

\bibitem{UDD}
N.~L. Johnson, A.~W. Kemp, and S.~Kotz.
\newblock {\em Univariate Discrete Distributions}.
\newblock John Wiley \& Sons, Ltd, third edition, 2005.
\newblock \href {https://doi.org/10.1002/0471715816} {\path{doi:10.1002/0471715816}}.

\bibitem{johnson}
O.~Johnson.
\newblock Log-concavity and the maximum entropy property of the {P}oisson distribution.
\newblock {\em Stochastic Processes and their Applications}, 117(6):791--802, 2007.
\newblock \href {https://doi.org/10.1016/j.spa.2006.10.006} {\path{doi:10.1016/j.spa.2006.10.006}}.

\bibitem{JK}
B.~Jørgensen and C.~C. Kokonendji.
\newblock Discrete dispersion models and their {T}weedie asymptotics.
\newblock {\em AStA Advances in Statistical Analysis}, 100(1):43–78, 2015.
\newblock \href {https://doi.org/10.1007/s10182-015-0250-z} {\path{doi:10.1007/s10182-015-0250-z}}.

\bibitem{kella}
O.~Kella and A.~Löpker.
\newblock On binomial thinning and mixing.
\newblock {\em Indagationes Mathematicae}, 34(5):1121--1145, 2023.
\newblock \href {https://doi.org/10.1016/j.indag.2022.09.003} {\path{doi:10.1016/j.indag.2022.09.003}}.

\bibitem{kemp2}
A.~W. Kemp and C.~D. Kemp.
\newblock An alternative derivation of the {Hermite} distribution.
\newblock {\em Biometrika}, 53(3--4):627--628, 1966.
\newblock \href {https://doi.org/10.1093/biomet/53.3-4.627} {\path{doi:10.1093/biomet/53.3-4.627}}.

\bibitem{kemp1}
C.~D. Kemp and A.~W. Kemp.
\newblock {Some properties of the `Hermite' distribution}.
\newblock {\em Biometrika}, 52(3--4):381--394, 1965.
\newblock \href {https://doi.org/10.1093/biomet/52.3-4.381} {\path{doi:10.1093/biomet/52.3-4.381}}.

\bibitem{nolan}
J.~P. Nolan.
\newblock {\em Univariate Stable Distributions: Models for Heavy Tailed Data}.
\newblock Springer Series in Operations Research and Financial Engineering. Springer, 2020.
\newblock \href {https://doi.org/10.1007/978-3-030-52915-4} {\path{doi:10.1007/978-3-030-52915-4}}.

\bibitem{renyi}
A.~Rényi.
\newblock A {Poisson-folyamat} egy jellemzése [{A} characterization of {Poisson} processes].
\newblock {\em A Magyar Tudományos Akadémia Matematikai Kutató Intézetének Közleményei}, 1(4):519--527, 1956.
\newblock URL: \url{https://real.mtak.hu/200653/}.

\bibitem{sibuya}
M.~Sibuya.
\newblock Generalized hypergeometric, digamma and trigamma distributions.
\newblock {\em Annals of the Institute of Statistical Mathematics}, 31(3):373--390, dec 1979.
\newblock \href {https://doi.org/10.1007/bf02480295} {\path{doi:10.1007/bf02480295}}.

\bibitem{steutel1}
F.~W. Steutel and K.~{van}~Harn.
\newblock Discrete analogues of self-decomposability and stability.
\newblock {\em The Annals of Probability}, 7(5):893--899, 1979.
\newblock \href {https://doi.org/10.1214/aop/1176994950} {\path{doi:10.1214/aop/1176994950}}.

\bibitem{steutel2}
F.~W. Steutel and K.~van Harn.
\newblock {\em Infinite Divisibility of Probability Distributions on the Real Line}.
\newblock CRC Press, 2003.
\newblock \href {https://doi.org/10.1201/9780203014127} {\path{doi:10.1201/9780203014127}}.

\bibitem{townes}
F.~W. Townes.
\newblock Broadly discrete stable distributions, 2025.
\newblock \href {https://arxiv.org/abs/2509.05497} {\path{arXiv:2509.05497}}.

\end{thebibliography}

\bigskip

\noindent University of Leeds, Leeds, UK

\noindent \texttt{m.aldridge@leeds.ac.uk}

\end{document}